\documentclass[12pt]{amsart}
\usepackage{bbm, amssymb, amsmath, amsfonts}
\usepackage{graphicx, epsfig}

\usepackage{xcolor}
 \makeindex

\def\Bbb{\mathbb}

\def\bR{{\Bbb R}}

\def\bpm{\begin{pmatrix}}
\def\epm{\end{pmatrix}}

\def\bee{\begin{enumerate}}
\def\ee{\end{enumerate}}

\newtheorem{thm}{Theorem}[section]

\newtheorem{prop}[thm]{Proposition}
\newtheorem{proposition}[thm]{Proposition}

\newtheorem{lemma}[thm]{Lemma}
\newtheorem{remark}[thm]{Remark}

\begin{document}

%check for unused citations
%\nocite{*}

\title[Estimates for convolution operators] {On the sharp estimates for convolution operators with oscillatory kernel}
\author[I. Ikromov]{Isroil A. Ikromov}
\address{Institute of Mathematics named after V.I. Romanovsky,
University Boulevard 15, 140104, Samarkand,  Uzbekistan} \email{{\tt
i.ikromov@mathinst.uz}}

 \author[I. Ikromova]{Dildora I. Ikromova}
\address{Samarkand State University, University Boulevard 15, 140104, Samarkand, Uzbekistan
} \email{{\tt ikromova\_89@mail.ru}}


%\thanks{2000 {\em Mathematical Subject Classification.}
%42B10, 42B20, 42B37}
%\thanks{{\em Key words and phrases.}
 % Convolution operator, hypersurface, oscillatory integral, singularity}
%\thanks {We acknowledge the support for this work by the Deutsche Forschungsgemeinschaft under  DFG-Grant MU 761/11-1 .}

\begin{abstract} {
In this article, we   study the convolution  operators $M_k$  with oscillatory kernel, which are related to solutions to the Cauchy problem for the strictly hyperbolic equations.  The operator  $M_k$ is associated to the characteristic hypersurfaces $\Sigma\subset \mathbb{R}^3$ of a hyperbolic  equation and smooth amplitude function, which is homogeneous  of order $-k$ for large values of the argument.
We   study  the convolution operators    assuming  that the corresponding amplitude function is contained in a sufficiently small conic  neighborhood of a given point $v\in \Sigma$ at which  exactly one of the principal curvatures of the surface $\Sigma$ does not vanish. Such surfaces exhibit singularities of type $A$ in the sense of Arnol'd's classification.  Denoting by $k_p$ the minimal   number such that $M_k$ is $L^p\mapsto L^{p'}$-bounded for $k>k_p,$ we show that the number $k_p$ depends on some discrete characteristics of the surface $\Sigma$.
}

 \end{abstract}
\maketitle

%\vfill\newpage

%\tableofcontents

\thispagestyle{empty}

\setcounter{equation}{0}
\section{Introduction}\label{introduction}
It is well known that  solutions to the Cauchy problem for strictly hyperbolic equation up to a smooth function can be written as a sum of convolution operators of the type:
\begin{eqnarray*}\nonumber
\mathcal{M}_k=F^{-1}[e^{it\varphi(\xi)}a_k]F,
\end{eqnarray*}
where  $F$ is the Fourier transform operator , $\varphi\in
C^{\infty}(\mathbb{R}^{\nu}\backslash\{0\})$
 is homogeneous of order one, $a_{k}\in
C^{\infty}(\mathbb{R}^{\nu}_\xi)$ is a homogeneous function of order $-k$ for large $\xi$.

After scaling arguments in the time $t>0$ the operator $\mathcal{M}_k$ is reduced to the following convolution operator:
\begin{equation}\label{convoper}
M_k=F^{-1}[e^{i\varphi(\xi)}a_k]F.
\end{equation}

Let $1\leq p\leq 2$ be a fixed number:  We consider the problem: {\it find the minimal  number $k(p)$  such that $M_{k}: L^{p}(\mathbb{R}^\nu)\rightarrow
L^{p'}(\mathbb{R}^\nu)$ is bounded for any $k>k(p)$. }

Note that if $a_k(\xi)=|\xi|^{-k}$ for large $\xi$ with $0<k<\nu$ and $\varphi\equiv0$ then the problem can be solved by using the classical Hardy-Littlewood-Sobolev's inequality.
 More precisely, due to the classical Hardy-Littlewood-Sobolev's inequality  if $k\ge 2n(1/p-1/2)$ then the operator \eqref{convoper} is bounded from $L^p(\mathbb{R}^\nu)$ to $L^{p'}(\mathbb{R}^\nu)$.
Moreover, if $a_k$ is a classical  symbol of PDO and $\varphi\equiv0$ then we dealt with $L^p(\mathbb{R}^\nu)\mapsto L^{p'}(\mathbb{R}^\nu)$ boundedness problem for   pseudo-differential operators (see \cite{sugumoto88}). It is well-known that if $a_k$ is a classical symbol of the PDO with order zero then  the corresponding  PDO is bounded on
$L^p(\mathbb{R}^\nu)$ for $1<p\le2$.

Further, we will assume that the function $\varphi$ preserves sign, e.g. we will assume that $\varphi(\xi)\neq0$ for any $\xi\in \mathbb{R}^\nu\setminus\{0\}.$
 Note that, due to the oscillation factor  for a wider range of the order $k$ of the symbol $a_k$ we get  the $L^p(\mathbb{R}^\nu)\mapsto L^{p'}(\mathbb{R}^\nu) $ boundedness of the operator \eqref{convoper}.

Next, without loss of generality we may and will assume that $\varphi(\xi)>0$ for any $\xi\neq0$.
Since  $\varphi$ is a smooth homogeneous function of order one, then, due to the Euler's homogeneity relation we have:
$$
\sum_{j=1}^{n}\xi_{j}\frac{\partial\varphi(\xi)}{\partial\xi_j}=\varphi(\xi),
$$
 and hence the set $\Sigma$ defined by the following
 $$\Sigma=\{\xi\in\mathbb{R}^\nu : \varphi(\xi)=1\}$$
  is a smooth  or an analytic hypersurface provided $\varphi$ is a smooth or a real analytic function respectively.

Further, we use notation:
\begin{equation}\label{(1.2)}
k_p:=k_p(\Sigma):=\inf_{k>0}\{k>0:  M_{k}\, \mbox{is}\, L^{p}(\mathbb{R}^\nu)\rightarrow
L^{p'}(\mathbb{R}^\nu)\, \mbox{bounded for any}\, a_k  \}.
\end{equation}
It turns out that the number $k_p(\Sigma)$ depends on  geometric properties of the  hypersurface
$\Sigma.$ More precisely, the number depends on behavior of the Fourier transform of measures supported on $\Sigma$. The monograph   \cite{Iosevich} contains many modern   results related to  the Fourier transform of surface carried measures.

M. Sugimoto \cite{sugumoto98} consider the problem for the case when $\Sigma\subset \mathbb{R}^3$ is an analytic   surface having at least one non-vanishing principal curvature at every point and obtain an upper bound for the number $k_p(\Sigma)$. More precisely, M.Sugimoto introduces three classes of hypersurfaces in $\mathbb{R}^3$ with at least one non-vanishing principal curvature. For each class he obtains an upper bound for the number  $k_p(\Sigma)$. Moreover, he suggested  examples for each classes showing sharpness of the bounds for that examples.

  The natural question is: {\it Whether the upper estimate for the number $k_p(\Sigma)$ given by M. Sugimoto is the sharp bound for each hypersurfaces of the appropriate classes ?}

We obtain the exact value of $k_p(\Sigma)$ improving the results proved by M. Sugimoto for arbitrary analytic hypersurfaces having at least one non-vanishing  principal curvature and  smooth hypersurfaces under the so-called $R-$ condition introduced in \cite{IMmon}.

Since $\Sigma$ is a compact hypersurface, then following M. Sugimoto it is enough to consider the local version of the problem.
More, precisely we may assume that the amplitude function
 $a_{k}(\xi)$ is concentrated in a sufficiently small conic neighborhood $\Gamma$ of a fixed point $v\in S^2$  (where
$S^2$ is the unit sphere centered at the origin of the space  $\mathbb{R}^3$) and
$\varphi(\xi)\in C^{\infty}(\Gamma)$. Fixing such a point $v\in \mathbb{R}^3$, let us define the following local exponent $k_p(v)$ associated to this point:
\begin{equation}\label{local}
k_p(v):=\inf_{k>0}\{k: \exists\Gamma, \,  M_{k}: L^{p}(\mathbb{R}^3)\mapsto
L^{p'}(\mathbb{R}^3)\, \mbox{is bounded, whenever}  \, supp(a_k)\subset \Gamma\}.
\end{equation}

Further, we  use the following standard notation, assuming $F$ being  a sufficiently smooth function:
$$
\partial^\gamma F(x):=\partial_1^{\gamma_1}\dots \partial_\nu^{\gamma_\nu}F(x):=\frac{\partial^{|\gamma|}F(x)}{\partial x_1^{\gamma_1}\dots\partial x_\nu^{\gamma_\nu}},
$$ where $\gamma=(\gamma_1,\dots, \gamma_\nu)\in \mathbb{Z}^\nu_+$ is a multiindex,  with $\mathbb{Z}_+:=\{0\}\cup \mathbb{N}$, and $|\gamma|:=\gamma_1+\dots+\gamma_\nu.$

Also, for the sake of being definite  we will assume that
$v=(0,0,1)$ and $\varphi(0,0,1)=1$. Then after possible a linear transform in the space $\mathbb{R}^3_\xi$, which preserves the point $v$, we may assume $\partial_1\varphi(0, 0, 1)=0$ as well as $\partial_2\varphi(0, 0, 1)=0$.
Thus, in a neighborhood of the point  $v$ the hypersurface $\Sigma$ is given as the graph of a smooth function:
$$\Sigma\cap\Gamma=\{\xi\in\Gamma:\varphi(\xi)=1\}=\{
(\xi_1, \xi_2, 1+\phi(\xi_1,\xi_2))\in \mathbb{R}^3:  (\xi_1,\xi_2)\in U\},$$
where
$U\subset \mathbb{R}^2$ is a sufficiently small neighborhood of the origin and, $\phi\in C^{\infty}(U)$ is a smooth function satisfying the conditions: $\phi(0,0)=0,
\nabla\phi(0,0)=0$
(compare with \cite{sugumoto98}).

Surely, similarly one can define $\Sigma$ in a neighborhood of the point $v=(0,\dots, 0, 1)\in \mathbb{R}^\nu$ as
the graph of a smooth function $\phi$ defined in a sufficiently small  neighborhood $U$ of the origin of $\mathbb{R}^{\nu-1}$.

Also, we will assume that  the function $\phi$ has a singularity of type  $A_n (1\le n\le \infty)$ at the origin (see \cite{agv} for definition of $A$ type singularities).  The last condition means that the hypersurface  $\Sigma$ has exactly  one  non-vanishing principal curvature at the point $v$, whenever $n\ge2$ in the case $\nu=3$.

We use the following Proposition \cite{IMmon}:

\begin{prop}\label{normform}
 Assume that  $\phi$ is a smooth function defined in a neighborhood of the origin of $\mathbb{R}^2$ satisfying the conditions: $\partial_2^2\phi(0,0)\neq0$ and also $ \partial^\gamma\phi(0, 0)=0$ for any $|\gamma|\le 2$ with $\gamma\neq(0, 2)$.

    Then,  $\phi$ can be written in the following form on a sufficiently small neighborhood of the origin:
    $$
    \phi(x_1,x_2)=b(x_1,x_2)(x_2-\psi(x_1))^2+b_0(x_1), \eqno(1.2.2)
    $$
    where $b, b_0$ and $\psi$ are smooth functions with $b(0, 0)\neq 0$. The function $\psi (b_0)$  can be written as $\psi(x_1) =x_1^m\omega(x_1) $ with $\omega(0)\neq0, \, m\ge2$ and $  (b_0(x_1)=x_1^n\beta(x_1)$, with $\beta(0)\neq0, \, n\ge2$)  unless $\psi(b_0)$ is a flat function.
\end{prop}

\begin{remark}
It is easy to show that the numbers $m, n$ are well-defined for arbitrary smooth function $\phi$ having $A$ type singularity (see \cite{sugumoto98} and also \cite{IMmon}). Moreover, to each point $v\in \Sigma$ of the surface with at least one non-vanishing principal curvature we can attach a pair $(m(v), n(v))$ due to the Proposition \ref{normform}.
\end{remark}

\subsection{Classes of hypersurfaces}

Following M. Sugimoto  \cite{sugumoto98} we can introduce the following classes of hypersurfaces: We say that $\Sigma$ is of type I with order $n$ if $b_0(x_1)=x_1^n\beta(x_1)$, where $\beta$ is a smooth function with $\beta(0)\neq0$; $\Sigma$ is of type II  with order $m$ if $b_0$ is a flat function at the origin and also $\psi(x_1)=x_1^m\omega(x_1)$, where $\omega$ is a smooth function with $\omega(0)\neq0$, and finally, $\Sigma$ is of   type III if both functions $\psi, \, b_0$  are flat  at the origin.

Further,  we will assume that if $\Sigma$ is a $C^{\infty}$
hypersurface of type II then $b_0 \equiv 0$. This condition agree with the so-called
" $R-$condition" introduced in the monograph \cite{IMmon}.

Actually, M. Sugimoto  obtained an upper bound for the number $k_p(v)$ and also he provided examples for each classes showing sharpness of the bound for that examples.

\subsection{The main results}\label{genconj}

In this paper we will prove the following statement, which is the main our result.

\begin{thm}\label{main} Let $\Sigma\subset \mathbb{R}^3$ be a smooth surface having at least one non-vanishing principal curvature at the point $(0, 0, 1)$ and $1\le p\le 2$ be a fixed number and also $(m, n)$ be the pair defined by the Proposition \ref{normform}. Then there exists a conic neighborhood $\Gamma$ of the point $v$ such that for any $a_k$ with $supp(a_k)\subset \Gamma$ the   following statements hold:\\
(i) If $2m\ge n$ then $k_p(v)=(5-\frac2n)(\frac1p-\frac12)$;\\
(ii) If $\Sigma$ is a smooth hypersurface satisfying the $R-$condition and $m\ge 3$ and also $2m<n\le \infty$  then
\begin{equation}\label{II case}
k_p(v)= \max\left\{\left(5-\frac{1}{m}\right)\left(\frac{1}{p}
 -\frac{1}{2}\right), \, \left(6-\frac{2(m+1)}{n}\right)\left(\frac{1}{p}-\frac{1}{2}\right)-\frac{1}{2}
 +\frac{m}{n}\right\}.
\end{equation}
\end{thm}

\begin{remark}\label{to main}
Note that in the case (i) formally it is possible $m=\infty$ e.g. the $\psi$ can be a flat function. M. Sugimoto \cite{sugumoto98} suggested the example:
\begin{equation}\label{examplesug}
\phi_I(y)=1-(y_2^2-y_1^n),
\end{equation}
which corresponds to the case (i), with $\psi(y_1)\equiv0$. From our results it follows that the Sugimoto result is sharp in that case.
Moreover, the Sugimoto result, for  a surface of the class I with order $n$, is sharp if and only if $2m \ge n$.

Note that the first case (i) is agree with  the so-called linearly adapted condition introduced in the monograph  \cite{IMmon} (see also \cite{IM-uniform}). Also note that under the linearly adapted case the sharp uniform estimate for the Fourier transform of measures gives the sharp bound for the exponent $p$ in the $L^p\mapsto L^2$ Fourier  restriction estimate. As had been shown in   \cite{IMmon} it is only the case.

If $n=\infty$ e.g. if  $b_0$ is a flat function at the origin then so is $\psi$, under the condition $2m\ge n$. Hence, the Sugimoto result is sharp in that case also, in other words,  his results are sharp for arbitrary smooth surface of  the class III.

On the other hand if $2m<n< \infty$ then the result of Sugimoto \cite{sugumoto98} is not sharp for the hypersurfaces $\Sigma$ of the class I.
Our results show that one can not be ignored influence of the number $m$ for the surfaces of the class I.

For the case $n=\infty$ e.g. for hypersurfaces of the class II M. Sugimoto obtained the sharp bound for a subclass of  analytic surfaces  of the class II.
  It turns out that the analogical result holds true for arbitrary analytic hypersurfaces of the class II and also for arbitrary smooth surfaces of the class II under the $R-$ condition.
 More precisely, from our result it follows that actually  the statement of the   Theorem 2 proved by M. Sugimoto in the paper \cite{sugumoto98}  (page no. 396) holds true for arbitrary analytic hypersurface having type II and also for analogical  smooth hypersurfaces under the R-condition.

\end{remark}

The paper organized as follows, in the next section 2 we give preliminary results on relations between decay rate of oscillatory integrals and upper estimates for the number $k_p(v)$. Then we obtain an upper bound for the number $k_p(v)$, for each class of surfaces in the section 3. Finally, in the section 4 we give a lower bounds for the number $k_p(v)$, which are agree with the upper bound.
The results of the last section 4
finish a proof of the main Theorem \ref{main}.

{\bf Conventions:}  Throughout this article, we shall use the variable constant notation,
i.e., many constants appearing in the course of our arguments, often denoted by
$c, C, \varepsilon, \delta$; will typically have different values at different lines. Moreover, we shall use symbols
such as $\sim, \lesssim;$ or $<<$ in order to avoid writing down constants, as explained in \cite{IMmon} (
Chapter 1). By $\chi_0$ we shall denote a non-negative smooth cut-off function on $\mathbb{R}^\nu$ with typically
small compact support which is identically $1$ on a small neighborhood of the origin, and also  $\chi_1(x):=\chi_0(x)-\chi_0(2x)$.

\section{Preliminaries}

Note that the  boundedness problem for the convolution operators is related to behaviour  of the following convolution kernel:
\begin{eqnarray}\nonumber
K_k:=F^{-1}(e^{i\varphi(\xi)}a_k(\xi)).
\end{eqnarray}
We define the Fourier operator and its inverse by the following \cite{stein-book}:
\begin{eqnarray}\nonumber
F(u)(\xi):=\frac1{\sqrt{(2\pi)^\nu}}\int_{\mathbb{R}^\nu} e^{ i\xi \cdot x}u(x)dx,
\end{eqnarray}
and
\begin{eqnarray}\nonumber
u(x):=\frac1{\sqrt{(2\pi)^\nu}}\int_{\mathbb{R}^\nu} e^{-i \xi\cdot x} F(u)(\xi)d\xi
\end{eqnarray}
respectively for a Schwartz function $u$, where $\xi\cdot x$ is the  usual inner product of the vectors $\xi$ and $x$. Then there are  defined for distributions by the standard arguments.

It is well known that (see \cite{sugumoto98}) the main contribution to $K_k$ gives points $x$ which belongs to a sufficiently small neighborhood of the set $-\nabla \varphi(supp(a_k)\setminus\{0\})$.

In the paper \cite{sugumoto98} had been shown relation between the boundedness of the convolution operator $M_{k}$ and behaviour of the following oscillatory integral:
\begin{eqnarray*}\nonumber
I(\lambda, z)=\int_{\mathbb{R}^{\nu-1}}e^{i\lambda(z\cdot x+\phi(x))}g(x)dx,
 (\lambda>0,\, z\in \mathbb{R}^{\nu-1}),
\end{eqnarray*}
where $ g\in C_{0}^{\infty}(U)$ and $U$  is a sufficiently small neighborhood of the origin.

More precisely there were proved the following statements \cite{sugumoto98}:

 \begin{proposition} \label{Sugi1}
 Let $q\geq 2$ and $\alpha\geq 0$.  Suppose for all $g\in C_{0}^ {\infty}(U)$ and $\lambda>1$,
\begin{equation}\label{averbound}
\|I(\lambda,\cdot)\|_{L^{q}(\mathbb{R}_{z}^{\nu-1})}\leq C_{g}\lambda^{-\alpha},
\end{equation}
where $C_{g}$ is independent of $\lambda$. Then
$K_{k}(\cdot):=F^{-1}[e^{i\varphi (\xi)}a_{k}(\xi)](\cdot)\in
L^{q}(\mathbb{R}^{\nu})$ and $M_{k}:L^{p}(\mathbb{R}^{\nu}) \rightarrow
L^{p'}(\mathbb{R}^{\nu})$ bounded for $p=\frac{2q}{2q-1}$, if
$k>\nu-\alpha-\frac{1}{q}$.
 \end{proposition}

Also,  M. Sugimoto proved another version of the Proposition \ref{Sugi1} in the case $q=\infty$. One can define
\begin{eqnarray*}\nonumber
K_{k, j} (x)=F^{-1}[e^{i\varphi(\xi)}a_k (\xi) \Phi_j (\xi)](x).
\end{eqnarray*}
Here $\{\Phi_j (\xi)\}_{j=1}^\infty$  is a Littlewood-Paley partition of unity which is used to
define the norm
\begin{eqnarray*}\nonumber
\|v\|_{B^s_{
p, q}}:= \left(\sum_{j=0}^\infty
(2^{ js} \|F^{-1}(\Phi_j (\xi)F(v) \|_{L^p})^q\right)^\frac1q
\end{eqnarray*}
of Besov space $B^s_{p, q}$ (see \cite{Bergh}).

\begin{proposition}\label{Sugi2}
 Let  $\alpha\ge0$. Suppose, for all $g\in C^\infty_0(U)$ and  $\lambda>1$,
 \begin{equation}\label{vander}
\|I(\lambda; \cdot)\|_{L^\infty(\mathbb{R}^{\nu-1}_z)} \le C_g  \lambda^{-\alpha},
\end{equation}
where $C_g$ is independent of $\lambda$. Then $\{K_{k, j} (x)]\}_j^\infty$ is bounded in $L^\infty(\mathbb{R}^\nu)$, if
$k=\nu-\alpha$. Hence $M_k$ is $L^p\mapsto L^{p'}$ bounded, if $k>(2\nu-2\alpha)(\frac1{p}-\frac12)$. This
inequality can be replaced by an equation, if $p\neq1$.
\end{proposition}

\section{An upper bound for the number $k_p(v)$}

Note that we dealt with two-dimensional oscillatory integral $I(\lambda, z)$ e. g. $\nu=3$.
 If $\phi$ has singularity of type $A_{n-1}$ with $2\le n\le \infty$ at the origin
 and $|z|>\delta$ (where $\delta$ is a fixed positive number) then the phase function $\phi(x_1, x_2)+x\cdot z$ has no critical points provided $U$ is a sufficiently small neighborhood of the origin and $g\in C_0^\infty(U)$.
 Therefore we can use integration by parts arguments and obtain:
  \begin{eqnarray*}\nonumber
|I(\lambda, z)|\lesssim \frac1{|z\lambda|^2},
\end{eqnarray*}
 which is better than wanted.

 Further, we will assume that $|z|<<1$ and $U$ is a sufficiently small neighborhood of the origin.
 Then
 we can use stationary phase method with $x_2$ variable and obtain:
\begin{eqnarray*}\nonumber
I(\lambda, z)=\frac{C}{\lambda^\frac12}\int_{\mathbb{R}}e^{i\lambda(\phi_1(x_1, z_2)+z_2x_1^m\omega(x_1)+z_1x_1))}g(x_2^c(x_1, z_2), x_1)dx_1+ R(\lambda, z),
\end{eqnarray*}
 where $R$ is a remainder term satisfying the estimate $|R(\lambda, z)|\lesssim \lambda^{-\frac32}$ and $x_2^c(x_1, z_2)$ is the unique critical point of the phase function with respect to $x_2$. Moreover the phase function
 $\phi_1(x_1, z_2)$ can be written as:
\begin{eqnarray*}\nonumber
 \phi_1(x_1, z_2)=z_2^2B(z_2)+z_2^2x_1q(x_1, z_2),
 \end{eqnarray*}
 where $B, q$ are smooth functions with $B(0)\neq0$ (see \cite{BIM22}).

Then by using the Van der Corput type lemma  \cite{arhipov}  (see also \cite{duistermaat} for estimates with  more general phase function) we see that   the estimate \eqref{vander} holds true with $\alpha=\frac12+\frac1n$. In this case we can use the Proposition \ref{Sugi2} and have the following upper  bound for $k_p(v)$:
\begin{equation}\label{upbound}
k_p(v)\le \left(5-\frac2n\right)\left(\frac1p-\frac12\right).
\end{equation}
This case includes also the class of surfaces  type III e.g. the case $n=\infty$.  Note that the upper bound \eqref{upbound} does not depend on the number $m$. It turns out that, it is the sharp bound for  the $k_p(v)$ under the condition $2m\ge n$.

Now, we consider the more subtle case $2m<n$. In this case we use the following Lemma (compare with the Theorem 2 of \cite{sugumoto98}):

\begin{lemma}\label{aux}
Let $\phi$ be a smooth function satisfying the $R-$condition,  in which $2m< n\le \infty$ and $m\geq 3$ and also $\varepsilon>0$ be a fixed positive number.  Then
the following estimate
$$\|I(\lambda,\cdot)\|_{L^{m+1}(\mathbb{R}^2)}\leq
C\lambda^{-(\frac{1}{2}+\frac{2}{m+1})+\varepsilon}.$$
holds true.
\end{lemma}

\begin{proof} As noted before, we may assume that $|z|<<1$. So, in order to prove the Lemma \ref{aux} we will show validity of the following estimate:
 \begin{eqnarray*}\nonumber
\|I(\lambda,\cdot)\|_{L^{m+1}(V)}\leq
C|\lambda|^{-(\frac{1}{2}+\frac{2}{m+1})+\varepsilon},
\end{eqnarray*}
where $V$ is a sufficiently small neighborhood of the origin.

Moreover, due to the stationary phase arguments it is enough to  estimate the integral
\begin{eqnarray*}\nonumber
I_1(\lambda, z):=\int_{\mathbb{R}}e^{i\lambda(\Phi_1(x_1, z)}g(x_1, x_2^c(z_2, x_1))dx_1=:\int_{\mathbb{R}}e^{i\lambda(\Phi_1(x_1, z)}a(x_1, z_2)dx_1,
\end{eqnarray*}
where we used notation:
 \begin{eqnarray*}\nonumber
 a(x_1, z_2):=g(x_1, x_2^c(z_2, x_1)), \,
 \Phi_1(x_1, z):=x_1^n\beta(x_1)+z_2x_1^m\omega(x_1)+z_2^2x_1q(x_1, z_2)+z_1x_1.
\end{eqnarray*}

First, we assume that  $\{|z_2|<\delta|z_1|^{\frac{n-m}{n-1}}\}$,
where $\delta$ is a sufficiently small fixed number, which will be defined later.

If $\{\lambda|z_1|^{\frac{n}{n-1}}\leq
1\}$, then by using van der Corpute Lemma \cite{arhipov} we obtain:
\begin{equation}\label{4}
|I_1|\lesssim \frac{1}{|\lambda|^{\frac{1}{n}}}\leq\frac{1}
{|\lambda|^\frac{1}{n}(\lambda|z_1|^\frac{n}{n-1})^{\frac{2}{m+1}-\frac{1}{n}}}=
\frac{1}{|\lambda|^{\frac{2}{m+1}}|z_1|^{\frac{2n-m-1}{(n-1)(m+1)}}}.
\end{equation}

We show that actually, the estimate  \eqref{4} holds true for
$\lambda|z_1|^{\frac{n}{n-1}}>1$, whenever  $\delta$ is a sufficiently small positive number.

Indeed, we use change of variables $x_1=|z_1|^{\frac{1}{n-1}}y_1$ in the integral $I_1$
and denoting $y_1$ again by $x_1$ obtain:
$$I_1=|z_1|^{\frac{1}{n-1}}\int e^{i\lambda|z_1|^{\frac{n}{n-1}}\Phi_2(x_1, z)}a
(|z_1|^{\frac{1}{n-1}}x_1, z_2)dx_1,$$ where
 $$\Phi_2(x_1, z)=\beta(|z_1|^{\frac{1}{n-1}}x_1) x_{1}^{n}+\frac{z_2}{|z_1|^{\frac{n-m}{n-1}}}x_{1}^{m}\omega
 (|z_1|^{\frac{1}{n-1}}x_1)+\frac{z_2^{2}}{|z_1|}x_{1}q(|z_1|^\frac{1}{n-1}x_1, z_2)+
 sgn(z_1)x_1.$$

Note that
\begin{eqnarray*}\nonumber
\frac{|z_2|}{|z_1|^{\frac{n-m}{n-1}}}\le \delta <<1 \quad \mbox{and also}\quad    \frac{z_2^2}{|z_1|}\le \delta^2 |z_1|^\frac{n-2m+1}{n-1}<<1.
\end{eqnarray*}

There exists a number  $N$ such that the phase function  $\Phi_2$ has no critical point on the set  $\{|x_1|\ge N\}$.
Take a smooth non-negative function  $\chi_0$  such that
$$
\chi_0(x)=
 \begin{cases}
 1, \, \mbox{for} \quad |x|\leq 1\\
0, \,\mbox{for}\quad |x|>2.
\end{cases}
$$

We write the integral  $I_1$ as the sum of  two integrals by using the function $\chi_0$:
 \begin{eqnarray*}\nonumber
 I_1=|z_1|^{\frac{1}{n-1}}\int e^{i\lambda |z_1|^{\frac{n}{n-1}}\Phi_2(x_1, z)}
 a(|z_1|^{\frac{1}{n-1}}x_1, z_2)\chi_0\left(\frac{x_1}{N}\right)dx_1 +\\  |z_1|^{\frac{1}{n-1}}
 \int e^{i\lambda|z_1|^{\frac{n}{n-1}}\Phi_2(x_1, z)}a(|z_1|^{\frac{1}{n-1}}x_1, z_2)
 \left(1-\chi_0\left(\frac{x}{N}\right)\right)dx_{1}=:I_{11}+I_{12}.
  \end{eqnarray*}

We can use integration by parts formula in the integral  $I_{12}$ and obtain:
$$|I_{12}|\leq\frac{c|z_1|^{\frac{1}{n-1}}}{|\lambda|z_1|^{\frac{n-1}{n}}|}\leq
\frac{c|z_1|^{\frac{1}{n-1}}}{|\lambda|z_1|^{\frac{n}{n-1}}|^{\frac{2}{m+1}}}=
\frac{c}{|\lambda|^{\frac{2}{m+1}}|z_1|^{\frac{2n-m-1}{(n-1)(m+1)}}}.$$
 It is what we need. Surely, it coincides with the estimate \eqref{4}.

 Now, we consider estimate for the integral  $I_{11}$. The phase function of the integral can be considered as a small perturbation of the function
   $\beta(0)x_{1}^{n}+sgn(z_1)x_{1}$. Hence, there exists a positive number  $\delta>0$ such that the function $\Phi_2(x_1, z)$  has only non-degenerate critical points, whenever the parameter $z$ satisfies the condition: $|z_2|< \delta|z_1|^{\frac{n-m}{n-1}}$. Therefore we can use Van der Corpute type estimate and obtain:
 $$|I_{11}|\leq \frac{c|z_1|^{\frac{1}{n-1}}}{|\lambda|^{\frac{1}{2}}|z_1|^
 {\frac{n}{2(n-1)}}}\leq \frac{c}{|\lambda|^{\frac{2}{m+1}}|z_1|^
 {\frac{2n-m-1}{(n-1)(m+1)}}}.$$
 This completes a proof of the estimate \eqref{4} in the considered case.

Now, suppose  $\{|z_1|^{\frac{n-m}{n-1}}\leq
\frac{1}{\delta}|z_2|\}.$

If $|z_2|^{\frac{n}{n-m}}|\lambda|\leq 1$
then again by using Van der Corpute type  estimate  we obtain:
$$|I_{1}|\leq\frac{c}{|\lambda|^{\frac{1}{n}}}\leq\frac{c}
{|\lambda|^{\frac{1}{n}}(|z_2|^{\frac{n}{n-m}}|\lambda|)^
{\frac{2}{m+1}-\frac{1}{n}}}=\frac{c}{|\lambda|^{\frac{2}{m+1}}|z_2|^
{\frac{2n-m-1}{(n-m)(m+1)}}}.$$

 Finally, we consider the case
$|z_2|^{\frac{n}{n-m}}|\lambda|>1$, where we use essentially induction arguments.
In this case, it is natural to use change of variables $x_1 \mapsto |z_2|^{\frac{1}{n-1}}x_1$ in the integral
$I_1$ and one obtains:
$$I_1 =|z_2|^{\frac{1}{n-1}}\int
e^{i\lambda|z_2|^{\frac{n}{n-m}}\Phi_2(x_1,  z)}a(|z_2|^{\frac{1}{n-m}}x_1, z_2)dx_1,$$
where
\begin{eqnarray*}\nonumber
\Phi_{2}(x_1, z):=x^{n}_{1}\beta(|z_2|^{\frac{1}{n-m}}x_1)+sgn(z_2)x_{1}^{m}
\omega(|z_2|^{\frac{1}{n-m}}x_{1})+\\ |z_2|^{2-\frac{n-1}{n-m}}sgn(z_2)
x_{1}q(|z_2|^{\frac{1}{n-m}}x_1,z_2)+\frac{z_1}{|z_2|^{\frac{n-1}{n-m}}}x_1.
\end{eqnarray*}

There exists a positive number  $N$ such that the phase function  $\Phi_2$ has no critical points on the set  $\{|x_1|\geq N\}$.
Again, as before  we write the integral
 $I_1$ as the sum of two integrals  $I_{11}, \, I_{12}$  given by the formulas:
\begin{eqnarray*}\nonumber
I_{11}=|z_2|^{\frac{1}{n-m}}\int e^{i\lambda|z_2|^{\frac{n}{n-m}}\Phi_{2}(x_1, z)}a
(|z_2|^\frac{1}{n-m}x_1)\chi_0\left(\frac{x}{N}\right)dx,\\
|I_{12}|=|z_2|^{\frac{1}{n-m}}\int e^{i\lambda|z_2|^{\frac{n}{n-m}}\Phi_{2}(x_1, z)}
a(|z_2|^{\frac{1}{n-m}}x_1)\left(1-\chi_0\left(\frac{x}{N}\right)\right)dx.
\end{eqnarray*}

For the integral $I_{12}$ we get:
$$|I_{12}|\leq\frac{c|z_2|^{\frac{1}{n-m}}}{|\lambda|z_2|^{\frac{n}{n-m}}|}
\leq\frac{c|z_2|^{\frac{1}{n-m}}}{|\lambda|z_2|^{\frac{n}{n-m}}|^{\frac{2}{m+1
}}}=\frac{c}{|\lambda|^{\frac{2}{m+1}}|z_2|^{\frac{2n-m-1}{(n-m)(m+1)}}},$$
because on the support of the amplitude function of the integral $I_{12}$  the function $\Phi_{2}(x_1, z)$ has no critical points.

Finally, we consider estimate for the integral $I_{11}$. Note that
$\xi_{1}=\frac{z_1}{|z_2|^{\frac{n-1}{n-m}}}\in[-\frac{1}{\delta},
\frac{1}{\delta}]$. Since the interval $[-\frac{1}{\delta},
\frac{1}{\delta}]$ is the compact set then the required estimate follows from the local estimates.  Let $\xi_{1}=\xi_{1}^{0}$ be a fixed point of the interval $[-\frac{1}{\delta},
\frac{1}{\delta}]$. Further,  suppose  that the parameter $\xi_1$ changes in a sufficiently small neighborhood of the point $\xi_{1}^{0}$. Then the phase function $\phi_2$ can be considered as a small perturbation of the function
$$x_{1}^{n}\beta(0)+sgn(z_2)x_{1}^{m}\omega(0)+\xi_{1}^{0}x_1.$$

If $\xi_{1}^{0}\neq0$, then the phase function has only singularities of type
 $A_{k}$ with $(k\leq2)$. If $\xi_{1}^{0}=0$, Then the phase function has singularities of type  $A_{m-1}$  at the origin. In particular, if $2\le m\le3$ then the phase function has only singularities of type $A_{k}$ with $k\le2$.  Therefore due to the Theorem 1 of the paper \cite{akr}  there exists a function $\psi(\xi_1,z_2)\in L^{\frac{2(m-1)}{m-2}-0}:=\cap_{p<\frac{2(m-1)}{m-2}} L^p$
 such that  the following estimate:
 \begin{equation}\label{last}
 |I_{11}|\leq\frac{|z_2|^{\frac{1}{n-m}}\psi\left(\frac{z_1}{|z_2|^{\frac{n-1}{n-m}}},
 z_2\right)}{\lambda^{\frac{1}{2}}|z_2|^{\frac{n}{2(n-m)}}}=
 \frac{\psi\left(\frac{z_1}{|z_2|^{\frac{n-1}{n-m}}},z_2\right)}{|\lambda|
 ^{\frac{1}{2}}|z_2|^{\frac{n-1}{2(n-m)}}}
 \end{equation}
holds true for the integral  $I_{11}$, whenever $m\ge3$. If $m=2$ then there exists a function $\psi(\xi_1,z_2)\in L^{4-0}$ such that the estimate
 \eqref{last} holds true with the function $\psi$.

On the other hand the Van der Corpute Lemma yields:
 $$|I_{11}|\leq\frac{c|z_2|^{\frac{1}{n-m}}}{|\lambda z_2^{\frac{n}{n-m}}|^
 {\frac{1}{m}}}.$$
 By interpolation the two bounds we get:
 $$|I_{11}|\leq\frac{c\psi\left(\frac{z_1}{|z_2|^{\frac{n-1}{n-m}}},z_2\right)^{\frac{2(m-1)}{(m-2)(m+1)}}}
 {\lambda^{\frac{2}{m+1}}|z_2|^{\frac{2n-m-1}{n-m}}}.$$

Thus, for the integral  $I_1$  we have the estimate:
 $$|I_1|\leq\frac{c\chi_{|z_2|<\delta|z_1|^{\frac{n-m}{n-1}}}(z_2)}
 {|\lambda|^{\frac{2}{m+1}}|z_1|^{\frac{2n-m-1}{(n-1)(m+1)}}}+
 \frac{c\chi_{\delta|z_1|^{\frac{n-m}{n-1}}<|z_2|}(z_1)\psi\left(\frac{z_1}
 {|z_2|^{\frac{n-1}{n-m}}},z_2\right)^{\frac{2(m-1)}{(m-2)(m+1)}}}{|\lambda|
 ^{\frac{2}{m+1}}|z_2|^{\frac{2n-m-1}{(n-m)(m+1)}}}=:\frac{\widetilde{\psi}
 (z_1,z_2)}{|\lambda|^{\frac{2}{m+1}}}.$$

 Now, we show that  $\widetilde{\psi}\in L^{m+1-0}(V)$.
Indeed, let $1<p<m+1$ be a fixed number. Then
 $$2\int_{|z_1|<1}\frac{dz_1}{|z_1|^{\frac{2n-m-1}{(n-1)(m+1)}}p}
 \int_{0}^{\delta|z_1|^{\frac{n-m}{n-1}}}dz_2= 4\delta\int_{0}^{1}
 \frac{dz_1}{z_{1}^{\frac{2n-m-1}{(n-1)(m+1)}p-\frac{n-m}{n-1}}}.$$
 Obviously, the last integral converges, whenever $p<m+1$.
 Moreover, for $1<p<m+1$
 \begin{eqnarray*}\nonumber
 \int_{V}\frac{\psi(\frac{z_1}{|z_2|^{\frac{n-1}{n-m}}},z_2)
 ^{\frac{2(m-1)p}{(m-2)(m+1)}}\chi_{c|z_1|^{\frac{n-m}{n-1}}<|z_2|}(z_1)}
 {|z_2|^{\frac{2n-m-1}{(n-m)(m+1)}p}}dz_{1}dz_2=\\
 =\int_{0}^{1}dz_{2}
 \frac{1}{|z_2|^{\frac{2n-m-1}{(n-m)(m+1)}p-\frac{n-1}{n-m}}}\int_{0}^
 {\delta^{\frac{n-1}{m-n}}}\psi^{\frac{2(m-1)}{(m-2)(m+1)}p}(\xi_1, z_2)d\xi_1
 \leq \\ c\int_{0}^{1}\frac{dz_2}{z_2^{\frac{2n-m-1}{(n-m)(m+1)}p-
 \frac{n-1}{n-m}}}<+\infty
 \end{eqnarray*}
 whenever  $p<m+1$.

 Summing obtained estimates we came to a proof of the Lemma \ref{aux}.

 Indeed, for the integral $I_1$ we have the following uniform estimate:
 \begin{eqnarray*}\nonumber
 |I_1|\lesssim\frac1{|\lambda|^\frac1n}.
 \end{eqnarray*}
 If $\varepsilon\ge \frac2m-\frac1n$ then the last estimate enough to have a proof of the Lemma   \ref{aux}.

 Suppose $0<\varepsilon< \frac2m-\frac1n$. Then we use the estimate
 \begin{eqnarray*}\nonumber
 |I_1|\lesssim\frac{\tilde\psi_1(z)}{|\lambda|^\frac2m},
 \end{eqnarray*}
 with $\tilde\psi_1\in L^{m+1-0}(V)$.

Then interpolating the last two inequalities we get:
\begin{eqnarray*}\nonumber
 |I_1|\lesssim\frac{\tilde \psi_1^{1-\theta}}{|\lambda|^{\frac{\theta}n+\frac{2(1-\theta)}m}},
 \end{eqnarray*}
where $0<\theta<1$. We can choose the number $\theta$ such that the following relation
 \begin{eqnarray*}\nonumber
 \frac{\theta}n+\frac{2(1-\theta)}m=\frac2{m}-\varepsilon \quad \mbox{or}\quad \varepsilon=\theta\left(\frac2m-\frac1n\right)>0.
 \end{eqnarray*}
 holds.
 Then surely, the inclusion $\tilde \psi_1^{1-\theta}\in L^{m+1}(V)$ is obviously valid.

  Analogical result holds true for the case $n=\infty$.

 Indeed, assume $n=\infty$ then $b_0$ is a flat function at the origin. By $R-$condition we have $b_0\equiv0$.
 So, we have
 \begin{eqnarray*}
 \phi (x_1, x_2)=b(x_1, x_2)(x_2-x_1^m\omega(x_1))^2.
 \end{eqnarray*}

In this case the phase function $\Phi_2$ has the form:
 \begin{eqnarray*}
 \Phi_1 (x_1, z)=z_2 x_1^m\omega(x_1)+z_2^2x_1q(x_1, z_2)+z_1x_1.
 \end{eqnarray*}

Then if $|z_1|\ge |z_2|$ then the phase function has no critical point in $x_1$. Then we can use integration by parts formula and  have
  \begin{eqnarray*}
|I_1|\lesssim \frac1{1+|\lambda z_1|}.
 \end{eqnarray*}
The last estimate yields
\begin{eqnarray*}
|I_1|\lesssim \frac1{\sqrt[4]{|z_1z_2|}\sqrt{|\lambda|}}.
 \end{eqnarray*}

Now, suppose $|z_1|\le |z_2|$. Then we can pull out $z_2$ and have the following estimate
 \begin{eqnarray*}
|I_1|\le  \frac{\psi \left(\frac{z_1}{z_2}, z_2\right)}{|z_2\lambda|^\frac12}=:\frac{\tilde\psi(z)}{|\lambda|^\frac12},
 \end{eqnarray*}
where $\tilde\psi\in L^{\frac{2(m-1)}{m-2}-0}(V)$ (see \cite{akr}), in the case $m\ge3$ and if $m=2$ then $\tilde\psi\in L^{4-0}(V)$ . Then we have a conclusion of the Lemma \ref{aux} as before.

 Which finishes a proof of the Lemma \ref{aux}.

\end{proof}

From the Lemma  \ref{aux} it follows the required upper bound for the number $k_p(v)$ in the case $2m<n$.
Indeed, first, we use the Proposition \ref{Sugi1} and  obtain $L^{p_0}\mapsto L^{p'_0}$ boundedness of the convolution operator $M_k$ with $k>\frac52-\frac3{m+1}$ for $p_0=\frac{2m+2}{2m+1}$. Also, we get $L^{p_1}\mapsto L^{p'_1}$ boundedness of the convolution operator with $k>\frac52-\frac1n$ for $p_1=1$ and also $L^{p_2}\mapsto L^{p'_2}$ boundedness of the convolution operator with $k=0$ for $p_2=2$.
Then by analytic interpolation of the obtained estimates we get the required upper bound for the number $k_p(v)$.

Further, we consider a lower bound for the number  $k_p(v)$.

\section{On the sharpness of results (a lower bound for the number $k_p(v))$}\label{sharp}

 \begin{thm}\label{thsharp}
  If   $2m\ge n$, then there exists an amplitude function $a_k$ such that   the associated  operator  $M_k$
  is not $ L^p(\mathbb{R}^3)\mapsto L^{p'}(\mathbb{R}^3)$ bounded, whenever
  $k<(5-\frac{2}{n})(\frac{1}{p}-\frac{1}{2})$.
 \end{thm}

\begin{proof}  Actually, we show that the sequence of functions suggested by M. Sugimoto in the paper \cite{sugumoto98} can be used to prove sharpness of the upper for $k_p(v)$ in the case (i).   Let us take a smooth function in $\mathbb{R}^3$ such that  $a_k(\xi)=|\xi|^{-k}$ for large $\xi$.
Following M. Sugimoto we introduce the function:
  $G(y)=1+\phi(y_1,y_2)-y\nabla\phi(y)$.
  Define non-negative functions
   $f(0)=g(0)=1$ concentrated in a sufficiently small neighborhood of the origin, and a smooth function with $\psi(1)=1$ and with support in a sufficiently small neighborhood of the point  $1$.

We set
\begin{eqnarray*}\nonumber
  u_{j}(x)=2^{j\left(\frac{5}{2}-\frac{1}{n}\right)\left(-\frac{1}{p'}\right)}F^{-1}
  (v_{j}(2^{-j}\xi))(x),
\end{eqnarray*}
where
\begin{eqnarray*}\nonumber
  v_{j}(\xi)=\frac{f\left(2^{\frac{j}{2}}\frac{\xi_{1}}{\varphi(\xi)}\right)
  g\left(2^{\frac{j}{n}}\frac{\xi_{2}}{\varphi(\xi)}\right)\psi(\varphi(\xi))|\xi|^k}
  {\varphi(\xi)^{2}G\left(\frac{\xi_1}{\varphi(\xi)},\frac{\xi_2}{\varphi(\xi)}\right)}\in C^\infty_0(\mathbb{R}^3).
\end{eqnarray*}

The sequence   $\{F^{-1}(v_{j}(2^{-\frac{j}{2}}\xi_1,2^{-\frac{j}{n}}\xi_{2},
  \xi_3))\}_{j=1}^{\infty}$ is bounded in  $L^{p}(\mathbb{R}^{3})$. Indeed,  the classical   Hausdorff-Young inequality yields:
  \begin{eqnarray*}\nonumber
  \|F^{-1}(v_{j}(2^{-j/2}\cdot, 2^{-j/n}\cdot, \cdot))\|_{L^{p}}\lesssim
  \|v_{j}(2^{-j/2}\cdot, 2^{-j/n}\cdot, \cdot)\|_{L^{p'}}.
  \end{eqnarray*}

On the other hand
 \begin{eqnarray*}\nonumber
 \|v_{j}(2^{-j/2}\cdot,2^{-j/n}\cdot,\cdot)\|_{L^{p'}}^{p'}=
\int f^{p'}\left(\frac{\xi_{1}}{\varphi(2^{-j/2}\xi_{1},2^{-j/n}\xi_2,\xi_3)}\right)
g^{p'}\left(\frac{\xi_{2}}{\varphi(2^{-j/2}\xi_{1},2^{-j/n}\xi_2,\xi_3)}\right)\\
\frac{\psi^{p'}(\varphi(2^{-j/2}\xi_{1},2^{-j/n}\xi_2,\xi_3))
((2^{-j/2}\xi_1)^2+(2^{-j/n}\xi_2)^2+\xi_{3}^{2})^{k/2}}{\varphi(2^{-j/2}
\xi_1,2^{-j/n}\xi_2,\xi_3)^{2p'}G^{p'}(\xi_{1}/\varphi(2^{-j/2}\xi_{1},2^{-j/n}\xi_2,\xi_3), \xi_{2}/\varphi(2^{-j/2}\xi_{1},2^{-j/n}\xi_2,\xi_3)) }d\xi.
 \end{eqnarray*}

Since  $\psi$ is supported in a sufficiently small neighborhood of  one, then we have:
$\frac{1}{2}\leq\varphi(2^{-j/2}\xi_1,2^{-j/n}\xi_2,\xi_3)\leq 2$. On the other hand supports of the functions  $f$ and  $g$ are  concentrated in a sufficiently small neighborhood of the origin. Hence,   $|\xi_1|<1$ and $|\xi_2|<1$ and also $|\xi_3|\sim1$, because  $\varphi(0, 0, 1)=1$.  This yields:
$$\|v_{j}(2^{-j/2}\cdot, 2^{-j/n}\cdot, \cdot)\|_{L^{p'}}\lesssim 1.$$
 Consequently,
 $$\|F^{-1}(v_{j}(2^{-j}\cdot))\|_{L^{p}}\lesssim 2^{j(\frac{1}{2}+
 \frac{n-1}{n}+1)\frac{1}{p'}}=2^{j(\frac{5}{2}-\frac{1}{n})\frac{1}{p'}}.$$
 Hence the sequence  $\{u_{j}\}_{j=1}^{\infty}$ is bounded in the space  $L^{p}(\mathbb{R}^3)$.

 On the other hand we have the relation:
  $$M_{k}u_{j}(x)=2^{j(\frac{5}{2}-\frac{1}{n})(-\frac{1}{p'})-kj+2j}
  F^{-1}
  \left(e^{i\varphi(\xi)}\frac{f\left(2^{\frac{j}2}\frac{\xi_1}
  {\varphi(\xi)}\right)g\left(2^{\frac{j}{n}}
  \frac{\xi_2}{\varphi(\xi)}\right)\psi(2^{-j}\varphi(\xi))}
  {\varphi(\xi)^{2}G\left(\frac{\xi_1}{\varphi(\xi)},\frac{\xi_2}{\varphi(\xi)}\right)}\right).$$

 We perform the change of variables given by the scaling  $2^{-j}\xi\mapsto\xi$ and obtain:
$$M_{k}u_{j}(x)=\frac{2^{j((\frac{5}{2}-\frac{1}{n})(-\frac{1}{p'})-k+3)}}{\sqrt{(2\pi)^3}}
\int_{\mathbb{R}^3} e^{2^{j}i(\varphi(\xi)- x\xi)}\frac{f\left(2^{\frac{j}{2}}\frac{\xi_1}
{\varphi(\xi)}\right)g\left(2^{\frac{j}{n}}\frac{\xi_2}{\varphi(\xi)}\right)\psi(\varphi(\xi))}
{\varphi^{2}(\xi)G\left(\frac{\xi_1}{\varphi(\xi)},\frac{\xi_2}{\varphi(\xi)}\right)}d\xi.$$

Then following M. Sugimoto we use change of variables  $\xi=(\lambda
y,\lambda(1+\phi(y)))$  and get:
$$M_{k}u_{j}(x)=\frac{2^{j((\frac{5}{2}-\frac{1}{n})(-\frac{1}{p'})-k+3)}}{\sqrt{(2\pi)^3}}\int
e^{i2^{j}\lambda(1-(x_{1}y_{1}+x_{2}y_{2}+x_{3}(1+\phi(y))))}
f(2^{\frac{j}2}y_1)g(2^{\frac{1}{n}}y_2)\psi(\lambda)d\lambda dy.$$

Finally, we use change of variables
$2^{\frac{j}{2}}y_1\mapsto y_1, \,
2^{j/n}y_2\mapsto y_2$ and obtain:
\begin{eqnarray*}\nonumber
M_{k}u_{j}(x)=2^{j((\frac{5}{2}-\frac{1}{n})(-\frac{1}{p'})-k-\frac{1}{2}
-\frac{1}{n}+3)}\int_{\mathbb{R}^3}
e^{2^{j}i\lambda((x_3-1)-2^{-\frac{j}{2}}y_{1}x_{1}-2^{-\frac{j}{n}}y_{2}x_{2}
-x_{3}\phi(2^{-\frac{j}2}y_1,
2^{-\frac{j}n}y_2))}\\ f(y_{1})g(y_2)\psi(\lambda)d\lambda dy.
\end{eqnarray*}

 If
$|x_{3}-1|\ll 2^{-j},\, |x_{1}|\ll 2^{-j/2}, \, |x_2|\ll
2^{\frac{-j(n-1)}{n}}$,  then the phase is the non-oscillating function, because
$\phi(2^{-j/2}y_{1},2^{-j/n}y_2)=o(2^{-j})$ provided the supports of $f, g$ are small enough.

Consequently, we have the following lower bound:
$$\|M_{k}u_{j}\|_{L^{p'}}\gtrsim 2^{j((\frac{5}{2}-\frac{1}{n})
(-\frac{1}{p'})-k+\frac{5}{2}-\frac{1}{n}-(\frac5{2}-\frac1{n})
(\frac{1}{p'}))}=2^{j((5-\frac{2}{n})(-\frac{1}{p'})+\frac{5}{2}-
\frac{1}{n}-k)}=2^{j((5-\frac{2}{n})(\frac{1}{p}-\frac{1}{2})-k)}.$$

Therefore, if
$k<k_p(\Sigma):=(5-\frac{2}{n})(\frac{1}{p}-\frac{1}{2}),$ then
$\|M_{k}u_{j}\|_{L^{p'}}\rightarrow \infty (\mbox{as}\, j\to+\infty)$. Thus, the operator
 $M_{k}:L^{p}(\mathbb{R}^3)\rightarrow L^{p'}(\mathbb{R}^3)$ is unbounded.

\end{proof}

 The Theorem \ref{thsharp} finishes a proof of the part (i) of the main Theorem \ref{main}.

Further, we consider the case
 $2m<n$.

\begin{remark}\label{nonadap}
The proof of the Theorem \ref{thsharp} shows that if $2m<n$ and $k<(5-\frac1m)(\frac{1}{p}-\frac{1}{2}),$ then
$\|M_{k}u_{j}\|_{L^{p'}}\rightarrow \infty (\mbox{as}\, j\to+\infty)$. Thus, the operator
 $M_{k}:L^{p}(\mathbb{R}^3)\rightarrow L^{p'}(\mathbb{R}^3)$ is an unbounded operator, whenever $k<(5-\frac1m)(\frac{1}{p}-\frac{1}{2})$.
 Indeed, we can repeat all arguments of the Theorem \ref{sharp} taking the function
 \begin{eqnarray*}\nonumber
  u_{j}(x)=2^{j\left(\frac{5}{2}-\frac{1}{2m}\right)\left(-\frac{1}{p'}\right)}F^{-1}
  (v_{j}(2^{-j}\xi))(x),
\end{eqnarray*}
with
\begin{eqnarray*}\nonumber
  v_{j}(\xi)=\frac{f\left(2^{\frac{j}{2}}\frac{\xi_{1}}{\varphi(\xi)}\right)
  g\left(2^{\frac{j}{2m}}\frac{\xi_{2}}{\varphi(\xi)}\right)\psi(\varphi(\xi))|\xi|^k}
  {\varphi(\xi)^{2}G\left(\frac{\xi_1}{\varphi(\xi)},\frac{\xi_2}{\varphi(\xi)}\right)}\in C^\infty_0(\mathbb{R}^3)
\end{eqnarray*}
for the case $2m<n$ and obtain the following lower bound:
\begin{equation}\label{upbound2}
k_p(v)\ge (5-\frac1m)(\frac1p-\frac12).
\end{equation}
for the number $k_p(v)$ whenever $2m<n$.
\end{remark}

Now, we prove the Theorem.
 \begin{thm}\label{NLA}
 If $2m<n$, and $m\geq3$ then
 \begin{equation}\label{NLAcase}
  k_p(v)=\max\left\{(5-\frac{1}{m})(\frac{1}{p}
 -\frac{1}{2}),(6-\frac{2(m+1)}{n})(\frac{1}{p}-\frac{1}{2})-\frac{1}{2}
 +\frac{m}{n}\right\}.
 \end{equation}
  \end{thm}

\begin{proof}

Since we already  got the upper bound for  $k_p(v)$, then it is enough to prove a lower bound for that number.

If $k<(5-\frac{1}{m})(\frac{1}{p}-\frac{1}{2})$,
 then the operator  $M_{k}$ is not $L^p(\mathbb{R}^3)\mapsto L^p(\mathbb{R}^3)$ bounded (see Remark \ref{nonadap}).

Assume  $k<(6-\frac{2(m+1)}{n})
 (\frac{1}{p}-\frac{1}{2})-\frac{1}{2}+\frac{m}{n}$.
We show that   $M_{k}$ is not $L^p(\mathbb{R}^3)\mapsto L^p(\mathbb{R}^3)$ bounded.

We a little modified the M. Sugimoto sequence  and consider the sequence
\begin{eqnarray*}
u_{j}=2^{-\frac{3j}{p'}+\frac{j(m+1)}{n}}F^{-1}(v_{j}(2^{-j}\cdot))(x),
\end{eqnarray*}
where
$$v_{j}(\xi)=f\left(2^{\frac{jm}{n}}\left(
\frac{\xi_2}{\varphi(\xi)}-\left(\frac{\xi_1}{\varphi(\xi)}\right)^{m}\omega
\left(\frac{\xi_1}{\varphi(\xi)}\right)\right)\right)\frac{g\left(2^{\frac{j}{n}}\frac{\xi_1}
{\varphi(\xi)}\right)\psi(\varphi(\xi))|\xi|^{k}}{\varphi^{2}(\xi)
G\left(\frac{\xi_{1}}{\varphi(\xi)},\frac{\xi_2}{\varphi(\xi)}\right)},$$ where
$f,g,\psi\in C_{0}^{\infty}(\mathbb{R})$ are non-negative smooth functions satisfying the conditions: $f(0)=g(0)=1$ and supports of
$f,g$  lie in a sufficiently small neighborhood of the origin of  $\mathbb{R}$. Suppose $0<c<<1$ is a fixed positive number (say $c=0.0001$) and $\psi$ is a non-negative smooth function concentrated in a sufficiently small neighborhood of the point $c$ and identically vanishes in a neighborhood of the origin and also $\psi(c)=1$,  (cf. \cite{sugumoto98}). Obviously
$v_{j}\in C_{0}^{\infty}(\mathbb{R}^{3})$ and
$\|v_{j}\|_{L^{p'}(\mathbb{R}^{3})}\sim2^{-j\frac{m+1}{p'n}},$ where the symbol $"\sim"$  means that there exit non-zero constants   $c_{1},c_{2}>0$ such that   $$c_{1}2^{-j\frac{m+1}{n}}\leq\int_{\mathbb{R}^3}|v_{j}(\xi)|^{p'}d\xi\leq
  c_{2}2^{-j\frac{m+1}{n}}.$$

Indeed, we use change of variables  $\xi=\lambda(y_{1}, y_{2}, 1+\phi(y_{1},y_{2}))$ in the integral
$\int_{\mathbb{R}^{3}}|v_{j}(\xi)|^{p'}d\xi$.  Note that on the support of $v_j$ make sense the change of variables, provided $j$ is big enough.
 Then
we get:
\begin{eqnarray*}\nonumber
\int_{\mathbb{R}^{3}}|v_{j}(\xi)|^{p'}d\xi=\int_{\mathbb{R}^3}f^{p'}
(2^{j\frac{m}{n}}(y_{2}-y_{1}^{m}\omega(y_1)))g^{p'}(2^{\frac{j}{n}}
y_1)\psi^{p'}(\lambda) \\ \lambda^{(k-2)p'+2}(y_{1}^{2}+y_{2}^{2}+
(1+\phi(y_{1},y_{2}))^{2})^{\frac{kp'}{2}}
 G^{2-p'}(y_{1},y_{2})
 dy_{1}dy_{2}d\lambda\sim 2^{-j\frac{m+1}{n}}.
 \end{eqnarray*}

Thus for large  $j$ we have
  $$\|u_{j}\|_{L^{p}(\mathbb{R}^{3})}\sim 1.$$
  Now, we consider the lower estimate for $\|M_{k}u_{j}\|_{L^{p'}(\mathbb{R}^{3})}.$

We have:
$$M_{k}u_{j}=F^{-1}e^{i\varphi(\xi)}a_{k}(\xi)Fu_{j}=2^{-\frac{3j}{p'}
+j\frac{m+1}{np'}}F^{-1}(e^{i\varphi(\xi)}a_{k}(\xi)v_{j}(2^{-j}\xi))(x).$$

We perform  change of variables given by the scaling  $2^{j}\xi\rightarrow \xi$ and obtain:
\begin{eqnarray*}\nonumber
M_{k}u_{j}(x)=\frac{2^{\frac{3j}{p}+\frac{j(m+1)}{np'}-kj}}{\sqrt{(2\pi)^3}}\int_{\mathbb{R}^{3}}
e^{i2^{j}(\varphi(\xi)-\xi
x)}\\ f\left(2^{\frac{jm}{n}}\left(\frac{\xi_2}{\varphi(\xi)}-\left(\frac{\xi_1}
{\varphi(\xi)}\right)^{m}\omega\left(\frac{\xi_1}{\varphi(\xi)}\right)\right)\right) \frac{g\left(2^{\frac{j}{n}
}\frac{\xi_1}{\varphi(\xi)}\right)\psi(\varphi(\xi))}{\varphi^{2}(\xi)
G\left(\frac{\xi_1}{\varphi(\xi)},\frac{\xi_2}{\varphi(\xi)}\right)}d\xi.
\end{eqnarray*}

Finally, we use change of variables  $\xi\rightarrow
\lambda(y_1,y_2,1+\phi(y_1,y_2))$.  Then we have:
\begin{eqnarray*}\nonumber
M_{k}u_{j}(x)=\frac{2^{\frac{3j}{p}+\frac{j(m+1)}{p'}-kj}}{\sqrt{(2\pi)^3}}\int_{\mathbb{R}^3}
e^{i2^{j}\lambda(1-
x_{3}-(y_{1}x_{1}+y_{2}x_{2}+x_{3}\phi(y_{1},y_{2})))}\times\\ \times
f(2^{\frac{jm}{n}}(y_{2}-y_{1}^{m}\omega(y_1)))g(2^{\frac{j}{n}}y_{1})
\psi(\lambda)d\lambda dy_1 dy_2.
\end{eqnarray*}

Now, we perform the change of variables
 $$y_{1}=2^{-\frac{j}{n}}z_{1}, \, y_2=y_{1}^{m}\omega
 (y_1)+2^{-j\frac{m}{n}}z_2.$$
  Then we get
  \begin{eqnarray*}\nonumber
  M_{k}u_{j}(x)=2^{\frac{3j}{p}+\frac{m+1}{np'}j-\frac{m+1}{n}j-kj}
   \int e^{i2^{j}\lambda\Phi_3(z, x, j)}f(z_2)g(z_1)\psi(\lambda)d\lambda dz_1 dz_2,
  \end{eqnarray*}
where
\begin{eqnarray*}\nonumber
\Phi_3(z, x, j):=1- x_3 -(2^{-\frac{j}{n}}x_1 z_1 +
  x_2 2^{-\frac{jm}{n}}z_{1}^{m}\omega(2^{-\frac{j}{n}}z_1)+z_2 2^
  {-\frac{jm}{n}}x_2+\\ x_3 2^{-\frac{2jm}{n}}z_{2}^{2}b(2^{-\frac{j}{n}}z_1,
  2^{-\frac{jm}{n}}(z_{1}^{m}\omega(2^{-\frac{j}{n}}z_1)+z_2))+2^{-j}z_{1}^{n}
  \beta(2^{-\frac{j}{n}}z_1)).
 \end{eqnarray*}

We use stationary phase method in $z_2$ assuming , $|1-x_3|<<2^{-j}, \,
  |x_1|<<2^{-\frac{n-1}{n}j}, \, |x_2|<<2^{-\frac{j(n-m)}{n}}$ and obtain:
  \begin{eqnarray*}\nonumber
  M_k u_{j}(x)=2^{j(\frac{3}{p}+\frac{m+1}{np'}-\frac1{n}-\frac12
  -k)}
  \left(\int_{\mathbb{R}^{2}}e^{i2^{j}\lambda \Phi_4}f(z_{2}^{c}(z_1, x_2))g(z_1)\psi(\lambda)d\lambda dz_1+O(2^{j(\frac{2m}n-1))}\right),
\end{eqnarray*}
where
  \begin{eqnarray*}\nonumber
\Phi_4:=\Phi_4(z_1, x, j):=1-x_3 -x_1 z_1
  2^{-\frac{j}{n}}+x_2 2^{-\frac{jm}{n}}z_{1}^{m}\omega(2^{-\frac{j}{n}}z_1
  )+ \\ 2^{-j}z_{1}^{n}\beta(2^{-\frac{j}{n}}z_1)+x_{2}^{2}2^{-\frac{2jm}{n}}
  B(z_1,x_2).
\end{eqnarray*}

From here we obtain the lower bound:
  $$\|M_{k}u_{j}\|_{L^{p'}(\mathbb{R}^{3})}\ge 2^{j(\frac{3}{p}+
  \frac{m+1}{np'}-\frac{1}{n}-\frac{1}{2}-\frac{1}{p'}(3-\frac{m+1}{n})-k)}c,$$
   where $c>0$ is a constant which does not depend on $j$.
   Thus if
   $k<(6-\frac{2(m+1)}{n})(\frac{1}{p}-\frac{1}{2})-\frac{1}{2}+\frac{m}{n}$
   then the operator $M_k$ is not $L^p(\mathbb{R}^3)\mapsto L^p(\mathbb{R}^3)$ bounded.

Analogical result holds true for the case $n=\infty$.

Thus if  $k<k_p(v)$ then the   $M_k$ is not $L^p- L^{p'}$ bounded operator.
   This completes a proof of the Theorem \ref{NLA}.

\end{proof}

The Theorem \ref{NLA} finishes a proof of the part (ii) of the main Theorem \ref{main}.

{\bf Acknowledgement:} The authors wish to thank to Professor Sh. A. Alimov  for valuable  discussions
of  results.


\begin{thebibliography}{99999999}

\bibitem{akr} Akramova D. I.  and Ikromov I. A., Randol Maximal Functions and the Integrability
of the Fourier Transform of Measures, Mathematical Notes, 2021, Vol. 109, No. 5, pp. 661-678.

\bibitem{arhipov}
 Arhipov, G.\,I.,  Karacuba, A.\,A.,   {\v{C}}ubarikov, V.\,N.,
\newblock Trigonometric integrals.
\newblock {\em Izv. Akad. Nauk SSSR Ser. Mat.}, 43 (1979), 971--1003, 1197 (Russian);
English translation in {\em Math. USSR-Izv.}, 15 (1980), 211--239.

\bibitem{agv}{
V. I. Arnol'd, S. M. Gusein-zade, and A. N. Varchenko, "Singularities of differentiable mappings,"  in
Classification of Critical Points of Caustics and Wavefronts 1985, Vol. 1, 1985, Birkh\"auser, Boston, Basel, Stuttgard.}

\bibitem{Bergh} J. Bergh and J. L\"ostr\"om, ``Interpolation Spaces,'' Springer-Verlag, Berlin/Heidelberg/New
York, 1976.


\bibitem{BIM22}  Estimates for maximal functions associated to hypersurfaces in $h<2$: Part II
A geometric conjecture and its proof for generic 2-surfaces.
http://arxiv.org/abs/2209.07352

 \bibitem{duistermaat}
 Duistermaat, I.\,I.,
\newblock Oscillatory integrals, {L}agrange immersions and unfolding of
  singularities.
\newblock {\em Comm. Pure Appl. Math.}, 27 (1974), 207--281.


\bibitem{IM-uniform}
Ikromov, I.\,A.,  M\"uller, D.,
\newblock Uniform estimates for the Fourier transform of surface carried measures in  $\bR^3$ and an application to Fourier restriction.
\newblock {\em    I. Fourier Anal. Appl.,} 17 (2011), no. 6, 1292--1332.




\bibitem{IMmon}
Ikromov, I.\,A.,  M\"uller, D.,
\newblock  Fourier restriction for hypersurfaces in three dimensions and Newton polyhedra;
  {\sl Annals of Mathematics Studies 194}, Princeton University Press, Princeton and Oxford 2016; 260 pp.

\bibitem{Iosevich} Iosevich A., Leflyand E.
Decay of the Fourier Transform: Analytic and Geometric Aspects, Birkh\"aser,  2014.


\bibitem{sugumoto88} M. Sugimoto, $L^{p}$-boundedness of Pseudo-differential operators
satisfying Besov estimates I, J. Math. Soc. Japan
Vol. 40, No. 1, 1988





\bibitem{sugumoto98} M. Sugimoto, Estimates for Hyperbolic Equations of Space Dimension 3,
Iournal of Functional  Analysi, 160, 382-407 (1998).


\bibitem{stein-book}
Stein, E.\,M.,
\newblock {Harmonic analysis: Real-variable methods, orthogonality, and
  oscillatory integrals}. {\em Princeton Mathematical Series} 43.
\newblock Princeton University Press, Princeton, NI, 1993.







\end{thebibliography}
\end{document}